\documentclass[12pt, reqno]{amsart}
\usepackage{amsmath, amsthm, amscd, amsfonts, amssymb, graphicx}
\usepackage[bookmarksnumbered, colorlinks, plainpages]{hyperref}
\input{mathrsfs.sty}

\newtheorem{theorem}{Theorem}[section]
\newtheorem{lemma}[theorem]{Lemma}

\newtheorem{corollary}[theorem]{Corollary}
\theoremstyle{definition}

\theoremstyle{remark}
\newtheorem{remark}[theorem]{Remark}

\numberwithin{equation}{section}

\begin{document}

\title{A Gr\"uss inequality for $n$-positive linear maps}
\author[M.S. Moslehian, R. Raji\'c]{Mohammad Sal Moslehian$^1$ and Rajna Raji\'c$^2$}

\address{$^1$ Department of Pure Mathematics, Center of Excellence in
Analysis on Algebraic Structures (CEAAS), Ferdowsi University of
Mashhad, P.O. Box 1159, Mashhad 91775, Iran.}
\email{moslehian@ferdowsi.um.ac.ir and moslehian@ams.org}

\address{$^2$ Faculty of Mining, Geology and Petroleum
Engineering, University of Zagreb, Pierottijeva 6, 10000 Zagreb,
Croatia} \email{rajna.rajic@zg.t-com.hr}



\begin{abstract}
Let $\mathscr{A}$ be a unital $C^*$-algebra and let $\Phi:
\mathscr{A} \to {\mathbb B}({\mathscr H})$ be a unital $n$-positive
linear map between $C^*$-algebras for some $n \geq 3$. We show that
$$\|\Phi(AB)-\Phi(A)\Phi(B)\|
\leq \Delta(A,||\cdot||)\,\Delta(B,||\cdot||)$$ for all operators
$A, B \in \mathscr{A}$, where $\Delta(C,\|\cdot\|)$ denotes the
operator norm distance of $C$ from the scalar operators.
\end{abstract} \maketitle

\section{Introduction}

Let ${\mathbb B}({\mathscr H})$ stand for the algebra of all
bounded linear operators on a complex Hilbert space $({\mathscr
H}, \langle \cdot, \cdot \rangle)$, let $\|\cdot\|$ denote the
operator norm and let $I$ be the identity operator. For
self-adjoint operators $A, B$ the order relation $A \leq B$ means
that $\langle A\xi,\xi\rangle\leq \langle
B\xi,\xi\rangle\,\,(\xi\in\mathscr{H})$. In particular, if $0 \leq
A$, then $A$ is called positive. If ${\rm dim}{\mathscr H}=k$, we
identify  ${\mathbb B}({\mathscr H})$ with the algebra
$\mathscr{M}_k$ of all $k \times k$ matrices with entries in
$\mathbb{C}$.

Let $\Delta(C,\|\cdot\|)=\inf_{\lambda\in\mathbb{C}}\|C-\lambda
I\|$ be the $\|\cdot\|$-distance of $C$ from the scalar operators.
It is known that $\Delta(C,\|\cdot\|)\leq \|C\|$ and
$\Delta(C,\|\cdot\|)=c(C)$ for any normal operator $C$, where
$c(C)$ denotes the radius of the smallest disk in the complex
plane containing the spectrum $\sigma (C)$ of $C$; see \cite{STA}.

A linear map $\Phi: {\mathscr A} \to {\mathscr B}$ between
$C^*$-algebras is said to be \emph{positive} if $\Phi(A) \geq 0$
whenever $A \geq 0$. Every positive linear map $\Phi$ satisfies
$\Phi(A^*)=\Phi(A)^*$ for all $A.$ We say that $\Phi$ is unital if
${\mathscr A}, {\mathscr B}$ are unital $C^*$-algebras and $\Phi$
preserves the identity. A linear map $\Phi$ is called
$n$-\emph{positive} if the map $\Phi_n: M_n({\mathscr A}) \to
M_n({\mathscr B})$ defined by $\Phi_n([a_{ij}]_{n \times
n})=[\Phi(a_{ij})]_{n \times n}$ is positive, where $M_n({\mathscr
A})$ stands for the $C^*$-algebra of $n \times n$ matrices with
entries in ${\mathscr A}$. It is known that $\|\Phi\|=1$ for any
unital $n$-positive linear map $\Phi$. $\Phi$ is said to be
\emph{completely positive} if it is $n$-positive for every $n\in
\mathbb{N}$. For a comprehensive account on completely positive
maps see \cite{PAU}.

The Gr\" uss inequality \cite{GRU}, as a complement of Chebyshev's
inequality, states that if $f$ and $g$ are integrable real functions
on $[a, b]$ and there exist real constants $\varphi, \phi, \gamma,
\Gamma$ such that $\varphi \leq f(x) \leq \phi$ and $\gamma \leq
g(x) \leq \Gamma$ hold for all $x \in [a, b]$, then
$$\left|\frac{1}{b-a} \int_a^b f(x)g(x)dx - \frac{1}{(b-a)^2} \int_a^b f(x)dx
\int_a^b g(x)dx \right|\leq \frac{1}{4}(\phi - \varphi)(\Gamma -
\gamma)\,.$$ This inequality has been investigated, applied and
generalized by many mathematicians in different areas of
mathematics; see \cite{DRA} and references therein. Peri\'c and
Raji\'c proved a Gr\" uss type inequality for unital completely
bounded maps \cite{P-R}.

In what follows $\mathscr{A}$ will stand for a unital
$C^*$-algebra. In this paper we prove that if $\Phi: \mathscr{A}
\to {\mathbb B}({\mathscr H})$ is a unital $n$-positive linear map
between $C^*$-algebras for some $n \geq 3$, then
\begin{eqnarray}
\|\Phi(AB)-\Phi(A)\Phi(B)\|
\leq \Delta(A,\|\cdot\|)\,\Delta(B,\|\cdot\|)
\end{eqnarray}
for all operators $A, B\in \mathscr{A}$.

\section{Main result}

To achieve our main result we need three lemmas. The first lemma can
be deduced from \cite[Theorem 2]{MAT} by adding a necessary
assumption $\Phi(A^*)=\Phi(A)^*$. We state it for the sake of
convenience.

\begin{lemma}\label{lemma1}
Let $\Phi: {\mathscr A} \to {\mathbb B}({\mathscr H})$ be a unital
$*$-$n$-positive \textup{(}not necessarily linear\textup{)} map
for some $n \geq 3$. Then
\begin{eqnarray}\label{cauchy}
\|\Phi(AB)-\Phi(A)\Phi(B)\|^2 \leq
\|\Phi(|A^*|^2)-|\Phi(A^*)|^2\|\,\|\Phi(|B|^2)-|\Phi(B)|^2\|
\end{eqnarray}
for all operators $A,B\in {\mathscr A}.$
\end{lemma}
\begin{proof}
Let $A$ and $B$ be two operators in ${\mathscr A}.$ We have
\begin{eqnarray*}
0 \leq \left[\begin{array}{cccccc} A^*\\B^*\\I\\0\\ \vdots\\0
\end{array}\right]\left[\begin{array}{cccccc}
A&B&I&0&\cdots&0
\end{array}\right]=\left[\begin{array}{cccccc}
A^*A & A^*B & A^*& 0&\cdots & 0\\
B^*A& B^*B & B^*& 0 &\cdots & 0\\
A& B& I& 0& \cdots&0\\
\vdots&\vdots & \vdots & \vdots&\ddots &\vdots\\
0& 0 & 0 &0&\cdots &0
\end{array}\right].
\end{eqnarray*}
Due to $\Phi$ is $n$-positive, we get
\begin{eqnarray*}
0 \leq \left[\begin{array}{cccccc}
\Phi(A^*A) & \Phi(A^*B) & \Phi(A^*)& 0&\cdots & 0\\
\Phi(B^*A)& \Phi(B^*B) & \Phi(B^*)& 0 &\cdots & 0\\
\Phi(A)& \Phi(B)& \Phi(I)& 0& \cdots&0\\
\vdots&\vdots & \vdots & \vdots&\ddots &\vdots\\
0& 0 & 0 &0&\cdots &0
\end{array}\right]\,,
\end{eqnarray*}
whence
\begin{eqnarray}\label{3x3}
0 \leq \left[\begin{array}{cccc}
\Phi(A^*A) & \Phi(A^*B) & \vline&\Phi(A)^*\\
\Phi(B^*A)& \Phi(B^*B) & \vline&\Phi(B)^*\\
\hline&&\vline\\
\Phi(A)& \Phi(B)&\vline& I
\end{array}\right].
\end{eqnarray}

It is known that the matrix $\left[\begin{array}{cc}
R&T\\T^*&S\end{array}\right]$ is positive if and only if $R,S$ are
positive and $R \geq TS^{-1}T^*$, where $S^{-1}$ denotes the
(generalized) inverse of $S$. Using this fact and noting to
\eqref{3x3} we get
\begin{eqnarray*}
\left[\begin{array}{cc}
\Phi(A^*A) & \Phi(A^*B) \\
\Phi(B^*A)& \Phi(B^*B)\\
\end{array}\right] \geq
\left[\begin{array}{cc}
\Phi(A)^*\\\Phi(B)^*\end{array}\right]\,I^{-1}\,\left[\begin{array}{cc}
\Phi(A)&\Phi(B)\end{array}\right] = \left[\begin{array}{cc}
\Phi(A)^*\Phi(A) & \Phi(A)^*\Phi(B) \\
\Phi(B)^*\Phi(A)& \Phi(B)^*\Phi(B)\\
\end{array}\right]
\end{eqnarray*}
or equivalently
\begin{eqnarray}\label{cov}
\left[\begin{array}{cc}
\Phi(A^*A)-\Phi(A)^*\Phi(A) & \Phi(A^*B)-\Phi(A)^*\Phi(B) \\
\Phi(B^*A)-\Phi(B)^*\Phi(A)& \Phi(B^*B)-\Phi(B)^*\Phi(B)\\
\end{array}\right] \geq 0
\end{eqnarray}
As noted in \cite{B-D}, the inequality \eqref{cov} implies that
\begin{eqnarray}\label{norm}
\|\Phi(A^*B)-\Phi(A)^*\Phi(B)\|^2 \leq
\|\Phi(|A|^2)-|\Phi(A)|^2\|\,\|\Phi(|B|^2)-|\Phi(B)|^2\|
\end{eqnarray}
Replacing $A$ by $A^*$ in \eqref{norm} we obtain \eqref{cauchy}.
\end{proof}

\begin{remark}
The inequality \eqref{cov} is known as the operator
covariance-variance inequality; see \cite{A-B-M} and references
therein for more information.
\end{remark}

The second lemma includes our main idea.

\begin{lemma}\label{lemma2}
Let $\Phi: \mathscr{A} \to {\mathbb B}({\mathscr H})$ be a unital
positive linear map between $C^*$-algebras. Then
\begin{eqnarray}\label{onlya2}
\|\Phi(A^*A)-\Phi(A)^*\Phi(A)\|\leq \Delta(A,\|\cdot\|)^2
\end{eqnarray}
for every normal operator $A \in \mathscr{A}$.
\end{lemma}
\begin{proof}
Let $A \in {\mathscr A}$ be a normal operator, and let $\mathscr{C}$
denote the $C^*$-algebra generated by $A, A^*$ and $I$. Then
$\mathscr{C}$ is commutative, so that the restriction of $\Phi$ to
$\mathscr{C}$ is completely positive by a known fact due to Choi;
see \cite{CHO, PAU}. The Stinespring dilation theorem states that
for any unital completely positive map $\Phi: {\mathscr C} \to
{\mathbb B}({\mathscr H})$ there exist a Hilbert space
$\mathscr{K}$, an isometry $V: {\mathscr H} \to {\mathscr K}$ and a
unital $*$-homomorphism $\pi: {\mathscr C} \to {\mathbb B}({\mathscr
K})$ such that $\Phi(A)=V^*\pi(A)V$. Now we have
\begin{eqnarray*}
\|\Phi(A^*A)&-&\Phi(A)^*\Phi(A)\|\\
&=& \|\Phi\big((A-\lambda I)^*(A-\mu I)\big)-(\Phi(A-\lambda
I))^*\Phi(A
-\mu I)\|\\
&=&\|V^*\pi\big((A-\lambda I)^*(A
-\mu I)\big)V - V^*(\pi(A-\lambda I))^*VV^*\pi(A-\mu I)V\|\\
&=&\|V^*(\pi(A-\lambda I))^*(I-VV^*)\pi(A-\mu I)V\|\quad\quad(\pi\textrm{~is~$*$-homomorphism})\\
&\leq& \|\pi(A-\lambda I)\|\,\|\pi(A-\mu I)\|\qquad\qquad\qquad\quad\quad (I-VV^*~\textrm{is~a~projection})\\
&\leq& \|A-\lambda I\|\,\|A-\mu I\|\qquad\qquad\qquad\qquad(\pi\textrm{~is~unital~and~norm~decreasing})\\
\end{eqnarray*}
for all $\lambda,\mu \in \mathbb C.$ From this it follows that
$$\|\Phi(A^*A)-\Phi(A)^*\Phi(A)\|\leq \inf_{\lambda\in \mathbb
C}\|A-\lambda I\|\,\inf_{\mu\in \mathbb C}\|A-\mu
I\|=\Delta(A,\|\cdot\|)^2.$$
\end{proof}

\begin{remark}
Passing the proof of previous lemma, one can easily deduce that in
the case of a unital completely positive map $\Phi: \mathscr{A} \to
{\mathbb B}({\mathscr H})$ the statement of Lemma~\ref{lemma2} is
valid for every $A\in {\mathscr A}.$
\end{remark}
The next lemma is well known.
\begin{lemma}\label{lemma3}
\cite[Theorem 1]{K-P} Let $A \in \mathscr{A}$ and $\|A\| < 1-(2/m)$
for some integer $m$ greater than $2$. Then there are $m$ unitary
elements $U_1, \dots, U_m \in {\mathscr A}$ such that
$mA=U_1+\cdots+U_m$.
\end{lemma}

We are ready to present our main result.

\begin{theorem}\label{main}
Let $\Phi: \mathscr{A} \to {\mathbb B}({\mathscr H})$ be a unital
$n$-positive linear map between $C^*$-algebras for some $n \geq 3$.
Then
\begin{eqnarray}\label{main2} \|\Phi(AB)-\Phi(A)\Phi(B)\| \leq
\Delta(A,\|\cdot\|)\,\Delta(B,\|\cdot\|)
\end{eqnarray}
for all operators $A, B \in \mathscr{A}$.
\end{theorem}

\begin{proof}
Using Lemmas \ref{lemma1} and \ref{lemma2} we obtain \eqref{main2}
for normal operators $A, B\in {\mathscr A}.$

\noindent Let $0 \neq A\in {\mathscr A}.$ Let $m>2$ be an integer
and $M=\frac{m^2+2}{m^2-2m}\|A\|$. Then
$$\|A/M\|= \frac{m^2-2m}{m^2+2}< 1-\frac{2}{m}\,.$$
By Lemma \ref{lemma3}, there are unitaries $U_1, \dots, U_m \in
{\mathscr A}$ such that $A=\frac{M}{m}\sum_{j=1}^{m}U_j$. Hence for
any normal operator $B \in {\mathscr A}$ we have
\begin{eqnarray*}
\|\Phi(AB)-\Phi(A)\Phi(B)\|&=& \frac{M}{m}\left\|\Phi\left(\sum_{j=1}^{m}U_jB\right)-
\Phi\left(\sum_{j=1}^{m}U_j\right)\Phi(B)\right\|\\
&\leq& \frac{m^2+2}{m^3-2m^2}\|A\|
\sum_{j=1}^{m}\|\Phi(U_jB)-\Phi(U_j)\Phi(B)\| \\
&\leq& \frac{m^2+2}{m^3-2m^2}\|A\|\sum_{j=1}^{m}\|U_j\|\|B\|
\qquad(\textrm{by~}\eqref{main2}~\textrm{for~normal~operators})\\
&\leq& \frac{m^2+2}{m^2-2m}\|A\|\,\|B\|.
\end{eqnarray*}
Letting $m\to \infty$, we infer that
\begin{eqnarray}\label{nov}
\|\Phi(AB)-\Phi(A)\Phi(B)\|\leq \|A\|\,\|B\|
\end{eqnarray}
for arbitrary $A$ and normal $B$. By repeating the same argument
for arbitrary $B$ and by using (\ref{nov}), we deduce that
$$\|\Phi(AB)-\Phi(A)\Phi(B)\|
\leq \|A\|\,\|B\|$$ for all $A, B \in \mathscr{A}$.

Next we observe that
\begin{eqnarray*}
\|\Phi(AB)-\Phi(A)\Phi(B)\| &=& \|\Phi\big((A-\lambda I)(B-\mu
I)\big)-\Phi(A-\lambda I)\Phi(B
-\mu I)\|\\
& \leq & \|A-\lambda I\|\,\|B-\mu I\|\\
\end{eqnarray*}
for all $\lambda,\mu\in \mathbb C.$ Thus,
$$\|\Phi(AB)-\Phi(A)\Phi(B)\|\leq \inf_{\lambda\in \mathbb
C}\|A-\lambda I\|\,\inf_{\mu\in\mathbb C}\|B-\mu
I\|=\Delta(A,\|\cdot\|)\,\Delta(B,\|\cdot\|)$$ for all $A,B\in
\mathscr A.$ This proves the theorem.
\end{proof}

\begin{remark}
Let us remark here that the inequality (\ref{main2}) does not have
to hold if $\Phi$ is assumed only to be a unital positive linear
map. To see this, let us choose $\Phi:\mathscr M_2\rightarrow
\mathscr M_2$ to be the transpose map. It is known (see e.g.
\cite{PAU}) that such a map is positive, but not $2$-positive.
Hence, $\Phi$ is not a $3$-positive map. Then, for matrices
$$A=\left[\begin{array}{cc}
1 & 2 \\
2 & 4
\end{array}\right]\,, \quad B=\left[\begin{array}{cc}
1 & 0 \\
0 & 4
\end{array}\right]
$$
we have
$$\|\Phi(AB)-\Phi(A)\Phi(B)\|=6.$$
Since $A$ and $B$ are positive matrices such that $\sigma
(A)=\{0,5\}$ and $\sigma (B)=\{1,4\},$ we conclude that
$\Delta(A,\|\cdot\|)=c(A)=\frac{5}{2}$ and
$\Delta(B,\|\cdot\|)=c(B)=\frac{3}{2}.$ Therefore,
$$\|\Phi(AB)-\Phi(A)\Phi(B)\|=6>\frac{15}{4}=\Delta(A,\|\cdot\|)\,\Delta(B,\|\cdot\|).$$

We do not know whether (\ref{main2}) holds if a map $\Phi$ is
assumed to be $2$-positive.
\end{remark}

Theorem \ref{main} extends the result obtained in \cite[Corollary
1]{P-R} to the case of $n$-positive linear maps, where $n\ge 3.$
Some applications of Theorem~\ref{main} on completely positive
maps were also given in \cite{P-R}. The first example of
$n$-positive map $(n\ge 2),$ which is not completely positive, was
obtained by Choi in \cite{CHO2}. He showed that the map $\Phi:
\mathscr M_k \to \mathscr M_k$ defined by
$$\Phi (T)=(k-1)\textup{tr}(T)I-T \quad (T\in \mathscr M_k)$$
is $(k-1)$-positive, but not $k$-positive. (Here 'tr' denotes the
trace.) Later, Takasaki and Tomiyama \cite{T-T} introduced a way
to construct new examples of $(k-1)$-positive linear maps from
$\mathscr M_k$ to $\mathscr M_k$ which are not $k$-positive.

In the next corollary we apply Theorem~\ref{main} on Choi's
$(k-1)$-positive linear map.

\begin{corollary}
Let $A,B\in \mathscr M_k,$ where $k\ge 4.$ Then
$$\|
(k^2-k-1)\textup{tr}(AB)I-kAB-(k-1)\textup{tr}(A)\textup{tr}(B)I
+\textup{tr}(B)A+\textup{tr}(A)B\|$$ $$\le
\frac{(k^2-k-1)^2}{k-1}\Delta(A,\|\cdot \|)\,\Delta(B,\|\cdot
\|).$$
\end{corollary}

\begin{proof}
Define
$$\Phi(T)=\frac{k-1}{k^2-k-1}\textup{tr}(T)I-\frac{1}{k^2-k-1}T
\quad (T\in \mathscr M_k).$$ By Theorem~\ref{main} we have
$$\|\Phi(AB)-\Phi(A)\Phi(B)\|\le
\Delta(A,\|\cdot\|)\,\Delta(B,\|\cdot\|)$$ since $\Phi$ is a
unital $(k-1)$-positive linear map. An easy computation shows that
$$\Phi(AB)-\Phi(A)\Phi(B)$$
$$=\frac{k-1}{(k^2-k-1)^2}[(k^2-k-1)\textup{tr}(AB)I-kAB-(k-1)\textup{tr}(A)\textup{tr}(B)I
+\textup{tr}(B)A+\textup{tr}(A)B],$$ from which the result
immediately follows.
\end{proof}


\end{document}